\documentclass[10pt]{amsart}

\usepackage{amssymb, latexsym, amsfonts, pigpen, enumitem, mathrsfs, amsthm, bbm, stackrel, hyperref}
\usepackage[matrix,arrow,curve]{xy}

\newtheorem{theorem}{Theorem}[section]
\newtheorem{lemma}[theorem]{Lemma}
\newtheorem{corollary}[theorem]{Corollary}
\newtheorem*{conjecture}{Conjecture}
\newtheorem*{theoremintr}{Theorem}
\newtheorem*{corollaryintr}{Corollary}

\theoremstyle{definition}
\newtheorem{definition}[theorem]{Definition}

\newtheorem{recollection}[theorem]{Recollection}
\newtheorem{notation}[theorem]{Notation}

\theoremstyle{remark}
\newtheorem{remark}[theorem]{Remark}

\numberwithin{equation}{section}

\newcommand{\mc}[1]{\mathcal{#1}}

\newcommand{\SL}{{\mathrm{SL}}}

\newcommand{\Hom}{{\mathrm{Hom}}}

\newcommand{\Z}{{\mathbb{Z}}}

\newcommand{\GW}{{\mathrm{GW}}}

\newcommand{\Vect}{{\rm Vect}}

\newcommand{\A}{\mathbb{A}}

\newcommand{\rank}{\operatorname{rank}}

\newcommand{\id}{\operatorname{id}}

\newcommand{\Gm}{{\mathbb{G}_m}}
\newcommand{\GL}{\mathrm{GL}}
\newcommand{\SH}{\mathcal{SH}}

\newcommand{\Ht}{\mathrm{H}}
\newcommand{\T}{\mathrm{T}}

\newcommand{\Spec}{\operatorname{Spec}}

\newcommand{\Sm}{\mathcal{S}m}

\newcommand{\et}{\mathrm{\acute{e}t}}

\newcommand{\QQ}{\mathbb{Q}}
\newcommand{\RR}{\mathbb{R}}
\newcommand{\CC}{\mathbb{C}}

\newcommand{\etale}{\'etale }
\newcommand{\Etale}{\'Etale }

\newcommand{\chark}{\operatorname{char}}

\newcommand{\SHk}{\mathcal{SH}(k)}
\newcommand{\Smk}{\mathcal{S}m_k}

\newcommand{\tensorunit}{\mathbf{1}}

\newcommand{\BettiR}{\mathrm{Re_{B\mathbb{R}}}}
\newcommand{\BettiC}{\mathrm{Re_{B\mathbb{C}}}}
\newcommand{\tp}{\mathrm{top}}
\newcommand{\crk}{\operatorname{rk_c}}

\newcommand{\pour}{\ar@{}[ur]|(0.2){\text{\pigpenfont G}}}
\newcommand{\podr}{\ar@{}[dr]|(0.2){\text{\pigpenfont A}}}

\newcommand{\leftrarrows}{\mathrel{\raise.75ex\hbox{\oalign{%
				$\scriptstyle\leftarrow$\cr
				\vrule width0pt height.5ex$\hfil\scriptstyle\relbar$\cr}}}}
\newcommand{\lrightarrows}{\mathrel{\raise.75ex\hbox{\oalign{%
				$\scriptstyle\relbar$\hfil\cr
				$\scriptstyle\vrule width0pt height.5ex\smash\rightarrow$\cr}}}}
\newcommand{\Rrelbar}{\mathrel{\raise.75ex\hbox{\oalign{%
				$\scriptstyle\relbar$\cr
				\vrule width0pt height.5ex$\scriptstyle\relbar$}}}}

\makeatletter
\def\leftrightarrowsfill@{\arrowfill@\leftrarrows\Rrelbar\lrightarrows}
\newcommand{\xleftrightarrows}[2][]{\ext@arrow 3399\leftrightarrowsfill@{#1}{#2}}
\makeatother

\begin{document}

\title[On the $\A^1$-Euler characteristic of the variety of maximal tori]{On the $\A^1$-Euler characteristic of the variety of maximal tori in a reductive group}


\author{Alexey Ananyevskiy}
\address{St. Petersburg Department, Steklov Math. Institute, Fontanka 27, St. Petersburg 191023 Russia}
\email{alseang@gmail.com}


\date{}

\thanks{}

\begin{abstract}
We show that for a reductive group $G$ over a field $k$ the $\A^1$-Euler characteristic of the variety of maximal tori in $G$ is an invertible element of the Grothendieck-Witt ring $\GW(k)$, settling the weak form of a conjecture by Fabien Morel. As an application we obtain a generalized splitting principle which allows one to reduce the structure group of a Nisnevich locally trivial $G$-torsor to the normalizer of a maximal torus.
\end{abstract}

\maketitle

\section{Introduction}
The $\A^1$-Euler characteristic $\chi^{\A^1}$ is an invariant which assigns to a smooth algebraic variety over a field $k$ an element of the Grothendieck--Witt ring of symmetric bilinear forms $\GW(k)$. This invariant refines the classical topological Euler characteristic in the sense that for a smooth algebraic variety $X$ over $\CC$ we have 
\[
\chi^{\A^1}(X)=\chi^{\tp}(X(\CC))
\]
while for a smooth algebraic variety $X$ over $\RR$ we have 
\[
\rank \chi^{\A^1}(X)=\chi^{\tp}(X(\CC)), \quad \operatorname{sgn} \chi^{\A^1}(X)=\chi^{\tp}(X(\RR))
\]
where $\rank$ and $\operatorname{sgn}$ are the rank and signature of a symmetric bilinear form. For a smooth variety $X$ over an arbitrary field $k$ at the current moment there is no general recipe how to compute its $\A^1$-Euler characteristic, although for projective varieties there are formulas in the terms of Hodge cohomology \cite[Theorem~1.3]{LR20} or Hochschild homology \cite[Theorem~1]{ABOWZ20}.

Besides being an interesting invariant by itself, $\chi^{\A^1}$ may be used to study fiber bundles. Marc Levine introduced a variant of Becker--Gottlieb transfers \cite{Lev19} in the motivic stable homotopy category $\SH(k)$ and showed that for a Nisnevich locally trivial fibration $p\colon Y\to X$ with the fiber $F$ the transfer map $\mathrm{Tr}_p\colon \Sigma^\infty_{\T} X_+\to \Sigma^\infty_{\T} Y_+$ gives a splitting for the morphism $\Sigma^\infty_{\T} p_+\colon \Sigma^\infty_{\T} Y_+\to \Sigma^\infty_{\T} X_+$ provided that $\chi^{\A^1}(F)$ is an invertible element of $\GW(k)$. The case of $F=G/H$ being a homogeneous space is of particular interest because this result may be used to obtain a variant of the splitting principle allowing to reduce $G$-torsors to $H$-torsors without loss of cohomological information. In this setting Marc Levine stated the following conjecture, attributing the strong form to Fabien Morel.
\begin{conjecture}[{\cite[Conjecture~2.1]{Lev19}}]
	Let $G$ be a split reductive group over a perfect field $k$, $T\le G$ be a split maximal torus and $N=N_G(T)$ be the normalizer of $T$ in $G$. Then
	\begin{itemize}[itemindent=6em]
		\item [{\normalfont\textbf{[Strong form]}}] $\chi^{\A^1}(G/N)=1$,
		\item [{\normalfont\textbf{[Weak form]}}] $\chi^{\A^1}(G/N)$ is an invertible element of $\GW(k)$.
	\end{itemize}
\end{conjecture}
\noindent
In the same paper Levine proved the strong form of the conjecture for $G=\GL_2$ and $\SL_2$ over a perfect field of characteristic not $2$ \cite[Proposition~2.6]{Lev19} and the weak form for $G=\GL_n$ and $\SL_n$ over a field of characteristic zero \cite[Corollary~2.4]{Lev19}. Then the strong form of the conjecture was proved in \cite[Theorem~1.6]{JP20} assuming that $\sqrt{-1}\in k$ and in case of $\chark k=p\neq 0$ additionally assuming $\Z[\tfrac{1}{p}]$-coefficients. In the current paper we prove the weak form of the conjecture over an arbitrary field (assuming $\Z[\tfrac{1}{p}]$-coefficients in positive characteristic) and also for non-split reductive groups, obtaining the following theorem. 

\begin{theoremintr}[{Theorem~\ref{thm:main}}]
	Let $G$ be a reductive group over a field $k$, $T\le G$ be a maximal torus and $N=N_G(T)$ be the normalizer of $T$ in $G$. Then the following holds.
	\begin{enumerate}
		\item 
		If $\chark k=0$ then $\chi^{\A^1}(G/N)$ is a unit.
		\item 
		If $\chark k=p\neq 0$ then $\chi^{\A^1}_{\Z[\frac{1}{p}]}(G/N)$ is a unit.
		\item 
		If $\chark k=p\neq 0$ and $G/N$ is strongly dualizable then $\chi^{\A^1}(G/N)$ is a unit.
	\end{enumerate}	
\end{theoremintr}
\noindent
Note that contrary to the original conjecture we do not assume $k$ to be perfect or $G$ to be split in the above theorem.

Combining this with the motivic Becker--Gottlieb transfers we obtain the following corollary.

\begin{corollaryintr}[{Corollary~\ref{cor:splitting}}]
	Let $k$ be a field and $\mathscr{G}\to X$ be a Nisnevich locally trivial $G$-torsor where $X\in \Smk$ and $G$ is a reductive group over $k$. Let $T\le G$ be a maximal torus and $N=N_G(T)$ be the normalizer of $T$ in $G$. Let $A^{*,*}(-)$ be a cohomology theory representable in $\SHk$ and suppose that one of the following holds:
	\begin{itemize}
		\item 
		$\chark k = 0$,
		\item 
		$\chark k = p\neq 0$ and $G/N$ is strongly dualizable or $A^{*,*}(X)$ has no $p$-torsion.
	\end{itemize}
	Then there exists $Y\in \Smk$ and a morphism $p\colon Y\to X$ such that
	\begin{enumerate}
		\item
		the $G$-torsor $p^*\mathscr{G}$ is induced from an $N$-torsor, 
		\item 
		the pullback homomorphism $p^A\colon A^{*,*}(X)\to A^{*,*}(Y)$ is split injective.
	\end{enumerate}
\end{corollaryintr}

\noindent This corollary may be seen as a generalized splitting principle. In particular, for $G=\GL_n$ the statement that a $G$-torsor is induced from an $N$-torsor corresponds to the statement that the vector bundle associated to the $G$-torsor admits a decomposition into a direct sum of line bundles on an $S_n$-Galois covering of the base with the structure group permuting the summands.

The proof of Theorem~\ref{thm:main} goes as follows. First we recall that the classical theory of symmetric bilinear forms yields that an element of the Grothendieck--Witt ring is invertible if and only if its rank and all of its signatures are units, so it is sufficient to treat the cases of algebraically closed and real closed fields. The rank of the $\A^1$-Euler characteristic can be computed using $\ell$-adic cohomology and the computation of $\mathrm{H}^*(G/N,\QQ_\ell)$ follows from the classical Borel description of the cohomology of $G/B$ with $B\le G$ being a Borel subgroup. In order to compute the signatures we reduce the problem to the case of $k=\RR$ using the fact that the classification of semisimple groups is the same over all real closed fields. Over $\RR$ one can use the topological Euler characteristic of the real Betti realization to compute the signature of the $\A^1$-Euler characteristic, so we study the real points of the variety $G/N$. They decompose as
\[
(G/N)(\RR)=\bigsqcup_i G(\RR)/N_i(\RR)
\]
where the $N_i=N_G(T_i)$ are the normalizers of the non-conjugate maximal tori $T_i\le G$. Then we show that 
\[
\chi^\tp(G(\RR)/N_i(\RR))=\left\{\begin{array}{ll}
	0, & \crk T_i\text{ is not maximal,} \\
	1, & \crk T_i\text{ is maximal}.
\end{array}  \right.
\]
Here $\crk T_i$ denotes the dimension of the maximal compact torus in $T_i(\RR)$ which equals to the dimension of the anisotropic part of $T_i$. Moreover, it can be shown that $\crk T_i$ is maximal only for one $i$ and this yields the claim of the theorem.

The paper is organized as follows. In Section~\ref{section:preliminary} we recall the notion of the $\A^1$-Euler characteristic and show that its invertibility may be detected by looking at the algebraic closure and real closures of the base field. In the next section we compute the rank of $\chi^{\A^1}(G/N)$ using $\ell$-adic cohomology. In Section~\ref{section:signature} we study the real points of $G/N$ and compute the signature of $\chi^{\A^1}(G/N)$. In the last section we prove the main theorem using all the previous computations and derive a generalized splitting principle.

{\it Acknowledgment.} I would like to thank the participants of the conference "Motivic Geometry" held online in September 2020 who found a gap in the previous version of the proof in the positive characteristic and Anastasia Stavrova for answering many questions on the general theory of reductive groups. I would also like to thank the referees for the careful reading of the paper and for all the suggestions that hopefully clarified the exposition. The author is a winner of the contest “Young Russian Mathematics”. The research was supported by the Russian Science Foundation grant 20-41-04401.

Throughout the paper we employ the following assumptions and notations.

\begin{tabular}{l|l}
	$\Smk$ & the category of smooth varieties over a field $k$\\
	$\SHk$ & the motivic stable homotopy category over a field $k$\\		
	$X_K$ & $X\times_{\Spec k}\Spec K$ for $X\in\Smk$ and a field extension $K/k$\\
	$X(\RR)$ & the $\RR$-points of $X$ viewed as a topological space with Euclidean topology \\
	$G$ & a reductive group over a field $k$\\		
	$T$ & a maximal torus in a reductive group $G$\\			
	$N=N_G(T)$ & the normalizer of the maximal torus $T$ in a reductive group $G$
\end{tabular}

\section{Preliminaries on the $\A^1$-Euler characteristic} \label{section:preliminary}

\begin{definition}[{\cite[\S 1]{DP84}}]
	Let $(\mc{C},\wedge,\tensorunit_{\mc{C}})$ be a symmetric monoidal category. $X\in \mc{C}$ is called \textit{strongly dualizable} if there exists $X^{\vee}\in\mc{C}$ and morphisms
	\[
	\mathrm{coev}\colon \tensorunit_{\mc{C}} \to X\wedge X^{\vee},\quad
	\mathrm{ev}\colon X^{\vee}\wedge X \to \tensorunit_{\mc{C}} 
	\]
	such that both compositions
	\begin{gather*}
		X \simeq \tensorunit_{\mc{C}} \wedge X \xrightarrow{\mathrm{coev}\wedge \id} X\wedge X^{\vee} \wedge X \xrightarrow{\id\wedge \mathrm{ev}} X\wedge\tensorunit_{\mc{C}} \simeq X, \\
		X^{\vee} \simeq  X^{\vee}\wedge \tensorunit_{\mc{C}} \xrightarrow{\id\wedge \mathrm{coev}} X^{\vee} \wedge X \wedge X^{\vee} \xrightarrow{\mathrm{ev} \wedge \id} \tensorunit_{\mc{C}} \wedge X^{\vee} \simeq X^{\vee}	
	\end{gather*}
	are equal to the identity. If $X$ is strongly dualizable then $X^{\vee}$, $\mathrm{coev}$ and $\mathrm{ev}$ are unique up to a unique isomorphism {\cite[\S 1]{DP84}}.
\end{definition}

\begin{definition}[{\cite[Defnition~4.1]{M01}}]
	Let $(\mc{C},\wedge,\tensorunit_{\mc{C}})$ be a symmetric monoidal category and let $X\in \mc{C}$ be strongly dualizable. The \textit{Euler characteristic} $\chi^{\mc{C}}(X)$ of $X$ is the composition
	\[
	\tensorunit_{\mc{C}} \xrightarrow{\mathrm{coev}} X\wedge X^{\vee} \xrightarrow{\simeq} X^{\vee} \wedge X \xrightarrow{\mathrm{ev}} \tensorunit_{\mc{C}}.
	\]
\end{definition}

\begin{remark}
	For the classical stable homotopy category $\SH$ and a topological space $X$ of finite homotopy type one has
	\[
	\chi^{\SH}(\Sigma^\infty_{S^1} X_+)= \chi^{\tp}(X) = \sum_{i} (-1)^i \dim \Ht^i(X,\QQ).
	\]
	This follows, for example, from Lemma~\ref{lem:chi_natural} below using the linearization functor $\SH\to \mathrm{D}(\QQ)$.
\end{remark}

\begin{lemma}[{\cite[Proposition~4.2]{M01}}] \label{lem:chi_natural}
	Let $F\colon (\mc{C},\wedge,\tensorunit_{\mc{C}}) \to (\widetilde{\mc{C}},\widetilde{\wedge},\tensorunit_{\widetilde{\mc{C}}})$ be a symmetric monoidal functor and let $X\in\mc{C}$ be strongly dualizable. Then $F(X)\in \widetilde{\mc{C}}$ is strongly dualizable and $\chi_{\widetilde{\mc{C}}}(F(X))=F(\chi_{\mc{C}}(X))$.
\end{lemma}

\begin{definition} \label{def:dual}
	Let $k$ be a field and $\SHk$ be the motivic stable homotopy category over $k$ \cite{MV99,V98}. We say that $X\in\Smk$ is
	\begin{itemize}
		\item \textit{strongly dualizable} if the suspension spectrum $\Sigma^\infty_{\T} X_+$ is strongly dualizable in $\SHk$,
		\item \textit{strongly dualizable with $\Z[\frac{1}{p}]$-coefficients} if $\Sigma^\infty_{\T} X_+$ is strongly dualizable in $\SH(k)[\frac{1}{p}]$, where $\SH(k)[\frac{1}{p}]$ is the respective localization of $\SHk$.
	\end{itemize}
\end{definition}

\begin{definition}
	Let $k$ be a field and $\GW(k)$ be the Grothendieck-Witt ring of $k$ that is the Grothendieck group of the semi-ring of
isometry clasess of non-degenerate symmetric bilinear forms over $k$. Below we will freely use the natural isomorphism of rings
	\[
	\Hom_{\SHk}(\Sigma^\infty_{\T} (\Spec k)_+,\Sigma^\infty_{\T} (\Spec k)_+)\simeq \GW(k)
	\]
	given by \cite[Corollary~1.25]{Mor12} and \cite[Theorem~6.4.1]{Mor04}. For a strongly dualizable $X\in\Smk$ we put 
	\[
	\chi^{\A^1}(X)=	\chi^{\SHk}(\Sigma^\infty_{\T} X_+) \in \GW(k)
	\]
	and refer to it as \textit{$\A^1$-Euler characteristic} of $X$.	For $X\in\Smk$ that is strongly dualizable with \mbox{$\Z[\frac{1}{p}]$-coefficients} we put 
	\[
	\chi^{\A^1}_{\Z[\frac{1}{p}]}(X)=	\chi^{\SHk[\frac{1}{p}]}(\Sigma^\infty_{\T} X_+) \in \GW(k)[\tfrac{1}{p}].
	\]
	We say that \textit{$\chi^{\A^1}(X)$ is a unit (resp. $\chi_{\Z[\frac{1}{p}]}^{\A^1}(X)$ is a unit)} if it is invertible as an element of the ring $\GW(k)$ (resp. $\GW(k)[\tfrac{1}{p}]$).	
\end{definition}

\begin{theorem} \label{thm:dualizable}
	For a field $k$ and $X\in\Smk$ we have the following.
	\begin{enumerate}
		\item \cite[Theorem~1.4]{Rio05} If $\operatorname{char} k=0$ then $X$ is strongly dualizable.
		\item \cite[Theorem~3.2.1]{EK20} If $\operatorname{char} k=p\neq 0$ then $X$ is strongly dualizable with $\Z[\frac{1}{p}]$-coefficients.		
	\end{enumerate}
\end{theorem}

\begin{definition}
	Let $k$ be a field. We say that $k$ is \textit{formally real} if $-1$ is not a sum of squares in $k$. We say that $k$ is \textit{real closed} if $k$ is formally real and for every algebraic extension $K/k$ with $K$ being formally real one has $k=K$.
\end{definition}

\begin{lemma} \label{lem:rankunit}
	Let $k$ be a field and $X\in\Smk$. Suppose that $k$ is not formally real.
	\begin{enumerate}
		\item
		If $\chark k= 0$ and $\rank \chi^{\A^1}(X)=\pm 1$ then $\chi^{\A^1}(X)$ is a unit.
		\item 
		If $\chark k=p\neq 0$ and $\rank \chi_{\Z[\frac{1}{p}]}^{\A^1}(X)=\pm 1$ then $\chi^{\A^1}_{\Z[\frac{1}{p}]}(X)$ is a unit.
		\item
		If $\chark k=p\neq 0$, $\rank \chi_{\Z[\frac{1}{p}]}^{\A^1}(X)=\pm 1$ and $X$ is strongly dualizable then $\chi^{\A^1}(X)$ is a unit.		
	\end{enumerate}	
\end{lemma}
\begin{proof}
	Since $k$ is not formally real then it follows from \cite[Theorem~V.8.9]{Bae78} combined with \cite[Lemma~V.7.7 and~Theorem~V.7.8]{Bae78} that the fundamental ideal 
	\[
	I(k)=\ker \left(\GW(k)\xrightarrow{\rank} \Z\right)=\ker \left(\mathrm{W}(k)\xrightarrow{\overline{\rank}} \Z/2\Z\right)
	\]
	is the nilradical of $\GW(k)$. Then in the case of $\chark k=0$ we have $\chi^{\A^1}(X)=\pm 1 + q$ with $q = \chi^{\A^1}(X) - \pm 1 \in I(k)$ being nilpotent whence the claim. If $\chark k=p\neq 0$ then $\chi_{\Z[\frac{1}{p}]}^{\A^1}(X)=\pm 1 + q$ for the same nilpotent $q\in I(k)[\tfrac{1}{p}]$ and the claim of item (2) also follows. If $X$ is strongly dualizable then Lemma~\ref{lem:chi_natural} yields $\rank \chi^{\A^1}(X)=\rank \chi_{\Z[\frac{1}{p}]}^{\A^1}(X)=\pm 1$ and the same reasoning applies.
\end{proof}

\begin{definition}
	Let $k=\overline{k}$ be an algebraically closed field and $\ell\neq \chark k$ be a prime. We denote
	\[
	\mathrm{Re}_{\QQ_\ell}\colon \SH(k) \to \mathrm{D}(\QQ_\ell)
	\]
	the contravariant \textit{$\ell$-adic realization functor} such that 
	\begin{enumerate}
		\item
		$\mathrm{Re}_{\QQ_\ell}$ is symmetric monoidal,
		\item
		for the total cohomology functor $\mathrm{D}(\QQ_\ell) \xrightarrow{\mathrm{H}^*} \Vect^*(\QQ_\ell)$, where $\Vect^*(\QQ_\ell)$ is the category of $\Z$-graded $\QQ_\ell$-vector spaces, the composition
		\[
		\Smk \xrightarrow{\Sigma^\infty_{\T}(-)_+} \SHk \xrightarrow{\mathrm{Re}_{\QQ_\ell}} \mathrm{D}(\QQ_\ell) \xrightarrow{\mathrm{H}^*} \Vect^*(\QQ_\ell)
		\]
		is given by the $\ell$-adic cohomology functor
		\[
		\mathrm{H}^*\left(\mathrm{Re}_{\QQ_\ell}\Sigma^\infty_{\T}(-)_+\right)\simeq 	\mathrm{H}^*(-,\QQ_\ell)=\varprojlim_n \mathrm{H}^*_{\et}(-,\Z/\ell^n\Z)\otimes_{\Z}\QQ.
		\]
	\end{enumerate}	
	In view of \cite[Corollary~2.39]{Rob15} (and standard properties of $\ell$-adic cohomology) these properties define the realization functor $\mathrm{Re}_{\QQ_\ell}$. Alternatively, $\mathrm{Re}_{\QQ_\ell}$ can be defined using the \etale realization functors of \cite{Ayo14} or \cite{CD16} combined with the appropriate change of topology and linearization functors.
\end{definition}

\begin{lemma} \label{lem:ladicrank}
	Let $k$ be a field, $\overline{k}$ its algebraic closure, $\ell\neq \chark k$ a prime and $X\in \Smk$.
	\begin{enumerate}
		\item 
		If $\chark k = 0$ then $\rank \chi^{\A^1}(X)= \sum_{i}  (-1)^i \dim_{\QQ_\ell} \mathrm{H}^i(X_{\overline{k}},\QQ_\ell)$.
		\item
		If $\chark k = p\neq 0$ then $\rank \chi^{\A^1}_{\Z[\frac{1}{p}]}(X)= \sum_{i}  (-1)^i \dim_{\QQ_\ell} \mathrm{H}^i(X_{\overline{k}},\QQ_\ell)$.
	\end{enumerate}
\end{lemma}
\begin{proof}
	Let $F\colon \SHk \to \SH(\overline{k})$ and $\mathrm{Re}_{\QQ_\ell}\colon \SH(\overline{k}) \to \mathrm{D}(\QQ_\ell)$ be the extension of scalars and the $\ell$-adic realization respectively. Suppose that $\chark k =0$, the proof in the other case is literally the same. Consider the following diagram.
	\[
	\xymatrix{
		\Hom_{\SH(k)}(\Sigma^\infty_{\T} (\Spec k)_+,\Sigma^\infty_{\T} (\Spec k)_+) \ar[r]^(0.75)\simeq \ar[d]_F
		& \GW(k) \ar[d]_f &\\
		\Hom_{\SH(\overline{k})}(\Sigma^\infty_{\T} (\Spec \overline{k})_+,\Sigma^\infty_{\T} (\Spec \overline{k})_+) \ar[r]^(0.75)\simeq  \ar[d]_{\mathrm{Re}_{\QQ_\ell}}
		& \GW(\overline{k}) \ar[r]_(0.6){\rank}^(0.6){\simeq} & \Z \ar[d]_i \\	
		\Hom_{\mathrm{D}(\QQ_\ell)}(\QQ_\ell,\QQ_\ell) \ar[rr]^\simeq & & \QQ_\ell	
	}
	\]
	Here $f$ is induced by the extension of scalars and $i$ is the inclusion. The top square commutes since the horizontal isomorphisms are functorial. The bottom part commutes since $\mathrm{Re}_{\QQ_\ell}$ is a unital homomorphism of rings and there exists a unique such homomorphism from $\Z$ to $\mathbb{Q}_\ell$. Then applying Lemma~\ref{lem:chi_natural} to the symmetric monoidal functors $F$ and $\mathrm{Re}_{\QQ_\ell}$ we get
	\[
	\rank \chi^{\A^1}(X)= (\mathrm{Re}_{\QQ_\ell}\circ F)(\chi^{\A^1}(X))=\chi^{\mathrm{D}(\QQ_\ell)}\left(\mathrm{Re}_{\QQ_\ell}(\Sigma^\infty_\T(X_{\overline{k}})_+)\right).
	\]	
	The claim follows since the Euler characteristic in the derived category is given by the alternating sum of cohomology group dimensions.
\end{proof}

\begin{lemma} \label{lem:formallyrealEuler}
	Let $k$ be a formally real field and $X\in\Sm_k$. Then $\chi^{\A^1}(X)$ is a unit if and only if $\chi^{\A^1}(X_K)$ is a unit for every field extension $K/k$ with $K$ being real closed.
\end{lemma}
\begin{proof}
	The "only if" part is clear from Lemma~\ref{lem:chi_natural} since the field extension induces a symmetric monoidal functor between the motivic stable homotopy categories. For the "if" part let $V$ be the set of all isomorphism classes of field extensions $K/k$ with $K$ being real closed. Recall that by \cite[Chapter~2, Corollary~4.8]{Sch85} for a real closed $K$ we have
	\[
	\GW(K) = \Z[\epsilon]/(\epsilon^2-1)
	\]
	where $\epsilon$ corresponds to the rank $1$ bilinear form $\langle -1 \rangle$, in particular, $q\in \GW(K)$ is invertible if and only if $q^2=1$. Consider the homomorphism
	\[
	\varphi \colon \GW(k) \to \prod_{K\in V} \GW(K)
	\]
	induced by field extensions. Let $q\in \GW(k)$ be such that $q_{K}$ is invertible for every $K\in V$, then by the above we have $\varphi(q^2)=1$ whence $q^2-1 \in \ker \varphi$. It follows from \cite[Chapter~2, Theorem~7.10 and Corollary~2.2]{Sch85} that $\ker \varphi$ is the nilradical of $\GW(k)$ thus $q^2=1+a$ for some nilpotent $a$. Then $q^2$ is invertible whence $q$ is invertible as well. Applying this to $q=\chi^{\A^1}(X)$ and using Lemma~\ref{lem:chi_natural} we get the claim.
\end{proof}


\begin{definition} \label{def:Betti}
	Let $\SH$ be the classical stable homotopy category. We denote
	\[
	\BettiC\colon \SH(\RR) \to \SH,\quad \BettiR\colon \SH(\RR) \to \SH
	\]
	the \textit{$\CC$-Betti and $\RR$-Betti realization functors} \cite[Section~3.3]{MV99} which for $X\in\Sm_\RR$ have
	\[
	\BettiC(\Sigma^\infty_{\T} X_+)= \Sigma^\infty_{S^1} X(\CC)_+,\quad \BettiR(\Sigma^\infty_{\T} X_+)= \Sigma^\infty_{S^1} X(\RR)_+.
	\]
	These functors are symmetric monoidal e.g. by \cite[Corollary~2.39]{Rob15}.
\end{definition}

\begin{remark}[{\cite[Remark~2.3]{Lev17}}] \label{rem:Betti}
	For $X\in \Sm_\RR$ one has
	\[
	\rank \chi^{\A^1}(X)=\BettiC(\chi^{\A^1}(X))=\chi^{\tp}(X(\CC)), \quad 
	\operatorname{sgn} \chi^{\A^1}(X)=\BettiR(\chi^{\A^1}(X))=\chi^{\tp}(X(\RR)).
	\]
\end{remark}

\section{$\A^1$-Euler characteristic of $G/N$: rank} \label{section:rank}

\begin{lemma} \label{lem:ladic}
	Let $k$ be an algebraically closed field, $G$ be a reductive group over $k$, $T\le G$ be a maximal torus and $N=N_G(T)$ be the normalizer of $T$ in $G$. Then for a prime $\ell\neq \chark k$ we have
	\[
	\mathrm{H}^i(G/N,\QQ_\ell)=\left\{
	\begin{array}{ll}
		\QQ_\ell, & i=0,\\
		0, & i\neq 0,
	\end{array} \right.
	\]
	where $\mathrm{H}^*(G/N,\QQ_\ell)=\varprojlim\limits_n \mathrm{H}^*_{\et}(G/N,\Z/\ell^n\Z)\otimes_{\Z}\QQ$ are the $\ell$-adic cohomology groups of $G/N$.
\end{lemma}
\begin{proof}
	Let $B\le G$ be a Borel subgroup such that $T\le B$, let $W=N/T$ be the Weyl group and let $\mathfrak{X}=\Hom(T,\Gm)$ be the lattice of characters. It is well known that $\ell$-adic cohomology groups of $G/B$ admit the so-called Borel presentation
	\[
	\varphi\colon S_{\QQ_\ell}(\mathfrak{X})/S_{\QQ_\ell}(\mathfrak{X})^W_{>0} \xrightarrow{\simeq} \mathrm{H}^*(G/B,\QQ_\ell),
	\] 
	where
	\begin{enumerate}
		\item $S_{\QQ_\ell}(\mathfrak{X})=Sym^*(\mathfrak{X})\otimes_\Z \QQ_\ell$ is the symmetric algebra of $\mathfrak{X}$ with $\QQ_\ell$-coefficients,
		\item $W$ acts on $S_{\QQ_\ell}(\mathfrak{X})$ via its canonical action on $\mathfrak{X}$,
		\item $S_{\QQ_\ell}(\mathfrak{X})^W_{>0}$ is the subspace generated by $W$-invariant homogeneous elements of positive degree,
		\item $\varphi$ is a homomorphism of algebras which in degree $1$ is given by
		\[
		\varphi(\chi)=c_1(L_\chi)\in \mathrm{H}^2(G/B,\QQ_\ell)
		\]
		for the line bundle $L_\chi$ over $G/B$ associated to the character $\chi$.
	\end{enumerate}
	This follows from the classical computation of rational Chow groups $\mathrm{CH}^*(G/B)\otimes \QQ$ (see e.g. \cite[4.6]{Dem74}) combined with the fact that since the Chow motive of $G/B$ is a sum of twisted Tate motives \cite[Theorem~2.1]{K91} and $\ell$-adic cohomology is a Weil cohomology theory then the cycle class map 
	\[
	\mathrm{CH}^*(G/B)\otimes \QQ_\ell \to  \mathrm{H}^{*}(G/B,\QQ_\ell)
	\]
	is an isomorphism.
	
	Consider the canonical morphism
	\[
	\pi\colon G/T\to G/B.
	\]
	Let $w_0$ be the longest element of the Weyl group and $U$ be the unipotent radical of $B$. It follows from the Bruhat decomposition \cite[Theorem~14.12]{Bor91} that over the big cell $Bw_0B/B\subseteq G/B$ morphism $\pi$ looks as follows.
	\[
	\xymatrix{
		\pi^{-1}(Bw_0B/B) \ar[r]^(0.55){=} \ar[r] \ar[d]^\pi &Bw_0B/T  \ar[d] & U\times B/T \ar[d]_{\pi_U} \ar[l]_(.45)\simeq\\
		Bw_0B/B \ar[r]^= & Bw_0B/B  & U \ar[l]_(0.35)\simeq	
	}
	\]
	Here $\pi_U$ is the projection onto the first factor. Recall that $B/T\cong U$ is a unipotent group thus $B/T$ is isomorphic as a variety to an affine space $\A^m$. Since $\pi$ is $G$-equivariant, $G$-translates of the big cell $Bw_0B/B$ give a Zariski open cover trivializing $\pi$ yielding that $\pi$ is a Zariski locally trivial bundle with the fibers being affine spaces. It follows that the pullback
	\[
	\pi^*\colon \mathrm{H}^*(G/B,\QQ_\ell) \to \mathrm{H}^*(G/T,\QQ_\ell)
	\]
	is an isomorphism.
	
	Combining the above we obtain that $\mathrm{H}^*(G/T,\QQ_\ell)$ admits the same Borel presentation as $\mathrm{H}^*(G/B,\QQ_\ell)$, i.e. that there is an isomorphism
	\[
	\psi\colon S_{\QQ_\ell}(\mathfrak{X})/S_{\QQ_\ell}(\mathfrak{X})^W_{>0} \xrightarrow{\simeq} \mathrm{H}^*(G/T,\QQ_\ell),\quad \psi(\chi)=c_1(L_\chi)\in \mathrm{H}^2(G/T,\QQ_\ell),
	\]
	where $L_\chi$ is the line bundle over $G/T$ associated to the character $\chi$. It is easy to see that $\psi$ is $W$-equivariant where the action of $W$ on $S_{\QQ_\ell}(\mathfrak{X})/S_{\QQ_\ell}(\mathfrak{X})^W_{>0}$ is induced by the action on $\mathfrak{X}$ as before while the action of $W$ on $\mathrm{H}^*(G/T,\QQ_\ell)$ comes from the right action of $W$ on $G/T$. Then \cite[Theorem B]{Che55b} implies that $\mathrm{H}^*(G/T,\QQ_\ell)$ is the regular representation of $W$. Thus the Hochschild-Serre spectral sequence
	\[
	E_1^{p,q}=\mathrm{H}^p(W,\mathrm{H}^q(G/T,\QQ_\ell))\Rightarrow \mathrm{H}^{p+q}(G/N,\QQ_\ell)
	\]
	associated to the Galois cover $G/T \to G/N$ degenerates yielding
	\[
	\mathrm{H}^*(G/N,\QQ_\ell) = (\mathrm{H}^*(G/T,\QQ_\ell))^W = \QQ_\ell. \qedhere
	\]
\end{proof}

\begin{corollary} \label{cor:chirank}
	Let $G$ be a reductive group over a field $k$, $T\le G$ be a maximal torus and $N=N_G(T)$ be the normalizer of $T$ in $G$. Then
	\begin{enumerate}
		\item
		$\rank \chi^{\A^1}(G/N)=1$ if $\chark k = 0$,
		\item
		$\rank \chi_{\Z[\frac{1}{p}]}^{\A^1}(G/N)=1$ if $\chark k = p\neq 0$.		
	\end{enumerate}
\end{corollary}
\begin{proof}
	Follows from Lemma~\ref{lem:ladicrank} and Lemma~\ref{lem:ladic}.
\end{proof}

\section{$\A^1$-Euler characteristic of $G/N$: signature} \label{section:signature}

\begin{notation}
	In this section we deal with reductive groups defined over $\RR$ and Lie groups and use the following notational conventions.
	\begin{itemize}
		\item 
		Reductive groups over $\RR$ are denoted by capital letters, e.g. $G$, $H$ and $T$.
		\item
		Lie groups are denoted by capital calligraphic letters, e.g. $\mathcal{G}$, $\mathcal{H}$, $\mathcal{K}$ and $\mathcal{T}$.
		\item
		A torus in $G$ is a torus in the scheme-theoretic sense.
		\item
		A compact torus in $\mathcal{K}$ is a compact torus in the Lie-theoretic sense, that is a connected abelian compact Lie subgroup.
		\item
		For a reductive group $G$ over $\RR$ and a subgroup $H\le G$ the normalizer $N=N_G(H)$ is the scheme-theoretic normalizer of $H$ in $G$.
		\item
		For a Lie group $\mathcal{G}$ and a subgroup $\mathcal{H}\le \mathcal{G}$ the normalizer $\mathcal{N}=N_\mathcal{G}(\mathcal{H})$ is the group-theoretic normalizer.
		\item
		For a Lie group $\mathcal{K}$ we denote $\mathcal{K}^o$ the connected component of the identity.
	\end{itemize}
\end{notation}	

\begin{recollection} \label{rec:reductive_and_Lie}
	We recall the following well-known facts on the connections between reductive groups over $\RR$ and Lie groups.
	\begin{enumerate}
		\item
		For a reductive group $H$ over $\RR$ the Lie group $H(\RR)$ has a finite number of connected components \cite[p.~121, Corollary~1]{PR94}. 
		
		\item
		For a reductive group $H$ over $\RR$ and a compact subgroup $\mathcal{K}\le H(\RR)$ there exists a maximal compact subgroup $\mathcal{K}\le \mathcal{K}'\le H(\RR)$. Moreover, all maximal compact subgroups in $H(\RR)$ are conjugate \cite[Proposition~3.10(2)]{PR94}. 
		
		\item
		Let $H_1\le H_2$ be reductive groups over $\RR$ and $\mathcal{K}_1\le H_1(\RR)$, $\mathcal{K}_2\le H_2(\RR)$ be maximal compact subgroups with $\mathcal{K}_1\le \mathcal{K}_2$. Then $H_1(\RR)\cap \mathcal{K}_2=\mathcal{K}_1$, the map
		\[
		i\colon \mathcal{K}_2/\mathcal{K}_1\to H_2(\RR)/H_1(\RR)
		\]
		is an embedding and $\mathcal{K}_2/\mathcal{K}_1$ is a deformation retract of $H_2(\RR)/H_1(\RR)$ \cite[Theorem~A and below]{Mos62}. In particular, $\mathcal{K}_2/\mathcal{K}_1$ and $H_2(\RR)/H_1(\RR)$ have the same homotopy type and $\chi^\tp(H_2(\RR)/H_1(\RR))=\chi^\tp(\mathcal{K}_2/\mathcal{K}_1)$.
		
		\item
		Let $H$ be a reductive group over $\RR$ and $\mathcal{K}\le H(\RR)$ be a compact subgroup. Then there exists an algebraic subgroup $K\le H$ such that $K(\RR)=\mathcal{K}$ \cite[Proposition~VI.5.2]{Che55a}. If $\mathcal{K}$ is a compact torus then $K$ may be chosen to be a torus.
	\end{enumerate}
\end{recollection}

\begin{definition}
	Let $G$ be a reductive groups over $\RR$ and $\mathcal{K}\le G(\RR)$ be a maximal compact subgroup. Recall that $\rank \mathcal{K}$ is the dimension of a maximal compact torus in $\mathcal{K}$. We put
	\[
	\crk G = \rank \mathcal{K}
	\]
	and refer to it as \textit{the compact rank of $G$}. This number does not depend on the choice of $\mathcal{K}$ since all the maximal compact subgroups in $G(\RR)$ are conjugate.
\end{definition}

\begin{remark}
	The compact rank $\crk G$ of $G$ coincides with the dimension of a maximal anisotropic torus in $G$. This can be seen using the fact that a torus $T$ is anisotropic if and only if $T(\mathbb{R})$ is compact \cite[Example~8.16.2]{Bor91} and that in this case $\dim T = \dim T(\mathbb{R})$.
\end{remark}

\begin{lemma} \label{lem:crank0}
	Let $G$ be a semisimple group over $\RR$. If $\crk G=0$ then $G$ is trivial.
\end{lemma}
\begin{proof}
	Suppose that $G$ is nontrivial and let $T\le G$ be a maximal torus. If $T$ is split then $G$ is split whence $G$ contains a copy of $\SL_2$ or $\mathrm{PSL}_2$ thus $\mathrm{SO}_2(\RR)\le \SL_2(\RR)\le G(\RR)$ or $\mathrm{PSO}_2(\RR)\le \mathrm{PSL}_2(\RR)\le G(\RR)$. Note that $\mathrm{SO}_2(\RR)\cong \mathrm{PSO}_2(\RR)$ is a compact torus of dimension $1$ which contradicts the assumption $\crk G=0$. If $T$ is not split then it contains an anisotropic part $T_a\le T$ and $T_a(\RR)$ is a compact torus of positive dimension.
\end{proof}

\begin{lemma} \label{lem:maximal_compact_tori}
	Let $G$ be a reductive group over $\RR$ and let $\mathcal{T}\le G(\RR)$ be a maximal compact torus. Then 
	\begin{enumerate}
		\item 
		there exists a unique maximal torus $T\le G$ such that $\mathcal{T}\le T(\RR)$,
		\item
		for $T$ as above one has $N_G(T)(\RR) = N_{G(\RR)} (\mathcal{T})$.
	\end{enumerate}	
\end{lemma}
\begin{proof}
	Let $\widetilde{T}\le G$ be the minimal torus such that $\mathcal{T}\le \widetilde{T}(\RR)$ (such a torus exists by item (4) of Recollection~\ref{rec:reductive_and_Lie}) and let $C=C_G(\widetilde{T})$ be its centralizer and $C'$ be the derived subgroup of $C$. We have $\crk C' =0$ otherwise there exists a compact torus $\mathcal{C}'\le C'(\RR)$ and $\mathcal{T}$ is not maximal. Then Lemma~\ref{lem:crank0} yields that $C'$ is trivial whence $C$ is a torus. Note that $C$ coincides with its centralizer whence it is a maximal torus. For every maximal torus $T\le G$ such that $\mathcal{T}\le T(\RR)$ one has $\widetilde{T}\le T$ whence $T\le C$ yielding $T=C$ and we obtain the claim of item (1). 
	
	For the item (2) note that the uniqueness of $T$ implies that $N_G(T)(\RR) \ge N_{G(\RR)} (\mathcal{T})$. The other inclusion follows from the fact that $\mathcal{T}$ is the unique maximal compact torus in $T(\RR)$.
\end{proof}

\begin{corollary} \label{cor:maximal_compact_tori_conj}
	Let $G$ be a reductive group over $\RR$ and $T_1,T_2\le G$ be maximal tori such that 
	\[
	\crk T_1 = \crk T_2 = \crk G.
	\]
	Then $T_1$ and $T_2$ are conjugate by an element of $G(\RR)$.
\end{corollary}
\begin{proof}
	Let $\mathcal{T}_1\le T_1(\RR)$ and $\mathcal{T}_2\le T_2(\RR)$ be maximal compact tori and let $\mathcal{T}_1\le \mathcal{K}_1\le G(\RR)$ and $\mathcal{T}_2\le \mathcal{K}_2\le G(\RR)$ be maximal compact subgroups. Since all maximal compact subgroups of $G(\RR)$ are conjugate and all the maximal compact tori in a compact Lie group are conjugate, there is a $g\in G(\RR)$ such that $g\mathcal{K}_2g^{-1} = \mathcal{K}_1$ and an $h\in \mathcal{K}_1$ such that $hg\mathcal{T}_2 g^{-1}h^{-1} = \mathcal{T}_1$. Then
	\[
	\mathcal{T}_1\le T_1(\RR),\quad \mathcal{T}_1=hg\mathcal{T}_2 g^{-1}h^{-1}\le (hgT_2g^{-1}h^{-1})(\RR)
	\]
	and Lemma~\ref{lem:maximal_compact_tori} yields that $T_1=hgT_2g^{-1}h^{-1}$.
\end{proof}

\begin{lemma} \label{lem:chi_orbit}
	Let $G$ be a reductive group over $\RR$, $T\le G$ be a maximal torus and $N=N_G(T)$ be the normalizer of $T$. Then
	\[
	\chi^\tp(G(\RR)/N(\RR))=\left\{\begin{array}{ll}
		0, & \crk T\neq \crk G \\
		1, & \crk T= \crk G.
	\end{array}  \right.
	\]
\end{lemma}
\begin{proof}
	Let $\mathcal{K}_{T}\le T(\RR)$, $\mathcal{K}_{N}\le N(\RR)$ and $\mathcal{K}_{G}\le G(\RR)$ be maximal compact subgroups with $\mathcal{K}_{T}\le \mathcal{K}_{N}\le \mathcal{K}_G$. By the item (3) of Recollection~\ref{rec:reductive_and_Lie} we have
	\[
	\chi^\tp(G(\RR)/N(\RR)) = \chi^\tp (\mathcal{K}_G/\mathcal{K}_{N}).
	\]
	
	Suppose that $\crk T\neq \crk G$. For the connected component $\mathcal{K}_{T}^o$ of $\mathcal{K}_{T}$ we have $\mathcal{K}_{T}^o\simeq (S^1)^{\times \crk T}$. Let $\mathcal{K}_{T}^o\le (S^1)^{\times \crk G} \le \mathcal{K}_G$ be a maximal compact torus in $\mathcal{K}_G$. Then 
	\[
	\mathcal{K}_G/\mathcal{K}_{T}^o \to \mathcal{K}_G/(S^1)^{\times \crk G}
	\]
	is a locally trivial bundle with the fiber $(S^1)^{\times (\crk G - \crk T)}$ and since $\chi^\tp(S^1)=0$ we have 
	\[
	\chi^\tp(\mathcal{K}_G/\mathcal{K}_{T}^o)=\chi^\tp(\mathcal{K}_G/(S^1)^{\times \crk G})\cdot \chi^\tp(S^1)^{\crk G - \crk T}=0.
	\]
	The map
	$
	\mathcal{K}_G/\mathcal{K}_{T}^o \to \mathcal{K}_G/\mathcal{K}_{T}
	$
	is a finite covering whence $\chi^\tp ( \mathcal{K}_G/\mathcal{K}_{T})=0$. Note that
	\[
	[N(\RR):T(\RR)]\le \#((N/T)(\RR))\le \#((N/T)(\CC))= \# W
	\]
	where $W$ is the Weyl group of $G_\CC$. It follows from item (3) of Recollection~\ref{rec:reductive_and_Lie} applied to $T\le N$ that
	\[
	\mathcal{K}_{N}/\mathcal{K}_{T}=N(\RR)/T(\RR)
	\]
	yielding $[\mathcal{K}_{N}:\mathcal{K}_{T}]< \infty$. Thus $\mathcal{K}_G/\mathcal{K}_{T} \to \mathcal{K}_G/\mathcal{K}_{N}$ is a finite covering and $\chi^\tp ( \mathcal{K}_G/\mathcal{K}_{N})=0$.
	
	Suppose that $\crk T= \crk G$. Then $\mathcal{T}=\mathcal{K}_T^o$ is a maximal compact torus in $G(\RR)$. Lemma~\ref{lem:maximal_compact_tori} yields that $N(\RR)=N_{G(\RR)}(\mathcal{T})$ whence
	\[
	\mathcal{K}_N=N(\RR)\cap \mathcal{K}_G = N_{\mathcal{K}_G}(\mathcal{T}).
	\]
	Since all maximal compact tori in $\mathcal{K}_G^o$ are conjugate, $N_{\mathcal{K}_G}(\mathcal{T})$ meets all connected components of $\mathcal{K}_G$. Whence
	\[
	\mathcal{K}_G/\mathcal{K}_N = \mathcal{K}_G / N_{\mathcal{K}_G}(\mathcal{T}) = \mathcal{K}_G^o / N_{\mathcal{K}_G^o}(\mathcal{T}).
	\]
	The cohomology ring $\mathrm{H}^*(\mathcal{K}_G^o/\mathcal{T},\RR)$ admits the so-called Borel presentation \cite[\S 26]{Bor53}, in particular, $\mathrm{H}^*(\mathcal{K}_G^o/\mathcal{T},\RR)$ is the regular representation of the Weyl group $W=N_{\mathcal{K}_G^o}(\mathcal{T})/\mathcal{T}$ \cite[Lemma~27.1]{Bor53}. Then applying the Leray-Serre spectral sequence for the cover 
	\[
	\mathcal{K}_G^o /\mathcal{T}\to \mathcal{K}_G^o / N_{\mathcal{K}_G^o}(\mathcal{T})
	\]
	with the deck transformation group $W$ we obtain 
	\[
	\mathrm{H}^i(\mathcal{K}_G^o/N_{\mathcal{K}_G^o}(\mathcal{T}),\RR) = \left\{
	\begin{array}{ll}
		\RR, & i=0,\\
		0, & i\neq 0.
	\end{array} \right.
	\]
	Summing up all the above we obtain
	\[
	\chi^\tp(G(\RR)/N(\RR)) = \chi^\tp (\mathcal{K}_G/\mathcal{K}_{N})= \chi^\tp (\mathcal{K}_G / N_{\mathcal{K}_G}(\mathcal{T})) = \chi^\tp (\mathcal{K}_G^o / N_{\mathcal{K}_G^o}(\mathcal{T})) = 1. \qedhere
	\]
\end{proof}

\begin{lemma} \label{lem:chitop}
	Let $G$ be a reductive group over $\RR$, $T\le G$ be a maximal torus and $N=N_G(T)$ be the normalizer of $T$. Then 
	\[
	\chi^{\tp}((G/N)(\RR))=1.
	\]
\end{lemma}
\begin{proof}
	There is a natural bijection 
	\[
	(G/N)(\RR) \simeq \{\,\text{maximal tori $T_i\le G$}\,\}
	\]
	that arises from the fact that over an algebraically closed field all maximal tori in $G$ are conjugate. This bijection is $G(\RR)$-equivariant with respect to the canonical action on $(G/N)(\RR)$ and the action by conjugation on the right-hand side. It follows from \cite[p.~121, Corollary~2 and its proof]{PR94} that
	\[
	(G/N)(\RR) = \bigsqcup_{i=1}^n X_i
	\] 
	is a disjoint union of finitely many $G(\RR)$-orbits $X_i$ and all these orbits are open and closed in $(G/N)(\RR)$. Additivity of the Euler characteristic yields
	\[
	\chi^\tp((G/N)(\RR)) = \sum_{i=1}^n \chi^\tp(X_i).
	\]
	From the above identification of $(G/N)(\RR)$ with the set of maximal tori we have
	\[
	X_i \simeq G(\RR)/N_i(\RR)
	\]
	where $N_i=N_{G}(T_i)$ are normalizers of some pairwise nonconjugate maximal tori $T_i\le G$. Lemma~\ref{lem:maximal_compact_tori} and Corollary~\ref{cor:maximal_compact_tori_conj} yield that there exists a unique up to conjugation maximal torus $\widetilde{T}\le G$ such that $\crk \widetilde{T} = \crk G$. Then the claim follows from Lemma~\ref{lem:chi_orbit}.
\end{proof}
\begin{corollary} \label{cor:signature}
	Let $G$ be a reductive group over $\RR$, $T\le G$ be a maximal torus and $N=N_G(T)$ be the normalizer of $T$. Then 
	\[
	\operatorname{sgn} \chi^{\A^1}(G/N)=1.
	\]
\end{corollary}
\begin{proof}
	Follows from Remark~\ref{rem:Betti} and Lemma~\ref{lem:chitop}.
\end{proof}

\section{$\A^1$-Euler characteristic of $G/N$ and splitting principle} \label{section:final}

\begin{theorem} \label{thm:main}
	Let $G$ be a reductive group over a field $k$, $T\le G$ be a maximal torus and $N=N_G(T)$ be the normalizer of $T$ in $G$. Then the following holds.
	\begin{enumerate}
		\item 
		If $\chark k=0$ then $\chi^{\A^1}(G/N)$ is a unit.
		\item 
		If $\chark k=p\neq 0$ then $\chi^{\A^1}_{\Z[\frac{1}{p}]}(G/N)$ is a unit.
		\item 
		If $\chark k=p\neq 0$ and $G/N$ is strongly dualizable then $\chi^{\A^1}(G/N)$ is a unit.
	\end{enumerate}
\end{theorem}
\begin{proof}
	If $k$ is not formally real then the claim follows from Lemma~\ref{lem:rankunit} and Corollary~\ref{cor:chirank}, so we suppose that $k$ is formally real. It follows from Lemma~\ref{lem:formallyrealEuler} that it is sufficient to treat the case of $k$ being real closed, so from now on we suppose $k$ to be real closed. For $\overline{G}=G/Z(G)$ and $\overline{T}=T/Z(G)$ we have 
	\[
	G/N = \overline{G}/N_{\overline{G}}(\overline{T})
	\]
	whence may assume that $G$ is semisimple.
	
	Let $\RR^{\mathrm{alg}}=\RR\cap \overline{\QQ}$ be the real closure of $\QQ$ in $\RR$. It follows from \cite[Theorem~III.2.1]{Sch85} that there is a unique isomorphism between $\RR^{\mathrm{alg}}$ and the real closure of $\QQ$ in $k$, in particular, there is a canonical embedding $\RR^{\mathrm{alg}}\subseteq k$. The classification of semisimple groups over real closed fields \cite[\S 4]{Gro72} yields that there exists a semisimple group $\widetilde{G}$ defined over $\RR^{\mathrm{alg}}$ such that $\widetilde{G}_k\cong G$. Let $\widetilde{T}\le \widetilde{G}$ be a maximal torus and $\widetilde{N}=N_{\widetilde{G}}(\widetilde{T})$ be the normalizer of $\widetilde{T}$ in $\widetilde{G}$. Applying Corollary~\ref{cor:chirank} and Corollary~\ref{cor:signature} we obtain that
	\[
	\rank \chi^{\A^1}(\widetilde{G}_\RR/\widetilde{N}_\RR) =1, \quad \operatorname{sgn}\chi^{\A^1}(\widetilde{G}_\RR/\widetilde{N}_\RR) =1
	\]
	whence $\chi^{\A^1}(\widetilde{G}_\RR/\widetilde{N}_\RR)=1$. The field extensions $k/\RR^{\mathrm{alg}}$ and $\RR/\RR^{\mathrm{alg}}$ induce isomorphisms
	\[
	\GW(k)\cong \GW(\RR^{\mathrm{alg}}) \cong \GW(\RR)
	\]
	by \cite[Chapter~2, Corollary~4.8]{Sch85}. Then Lemma~\ref{lem:chi_natural} applied to the symmetric monoidal functors $\SH(\RR^{\mathrm{alg}})\to \SH(\RR)$ and $\SH(\RR^{\mathrm{alg}})\to \SHk$ yields 
	\[
	\chi^{\A^1}(\widetilde{G}_k/\widetilde{N}_k) = \chi^{\A^1}(\widetilde{G}_\RR^{\mathrm{alg}}/\widetilde{N}_\RR^{\mathrm{alg}}) = \chi^{\A^1}(\widetilde{G}_\RR/\widetilde{N}_\RR)=1.
	\]
	Recall that over an algebraically closed field all the maximal tori in a semisimple algebraic group are conjugate. Whence over any field the variety $G/N$ does not depend on the choice of a maximal torus $T$. Then it follows from $G\cong \widetilde{G}_k$ that $G/N \cong \widetilde{G}_k/\widetilde{N}_k$ and
	\[
	\chi^{\A^1}(G/N)=\chi^{\A^1}(\widetilde{G}_k/\widetilde{N}_k)=1.\qedhere
	\]
\end{proof}

Combining this with the motivic Becker--Gottlieb transfers from \cite{Lev19} we obtain the following corollary.

\begin{corollary} \label{cor:splitting}
	Let $k$ be a field and $\mathscr{G}\to X$ be a Nisnevich locally trivial $G$-torsor where $X\in \Smk$ and $G$ is a reductive group over $k$. Let $T\le G$ be a maximal torus and $N=N_G(T)$ be the normalizer of $T$ in $G$. Let $A^{*,*}(-)$ be a cohomology theory representable in $\SHk$ and suppose that one of the following holds:
	\begin{itemize}
		\item 
		$\chark k = 0$,
		\item 
		$\chark k = p\neq 0$ and $G/N$ is strongly dualizable or $A^{*,*}(X)$ has no $p$-torsion.
	\end{itemize}
	Then there exists $Y\in \Smk$ and a morphism $p\colon Y\to X$ such that
	\begin{enumerate}
		\item
		the $G$-torsor $p^*\mathscr{G}$ is induced from an $N$-torsor, i.e. there exists an $N$-torsor $\mathscr{N}\to Y$ such that
		\[
		\mathscr{N}\times^N G \cong p^*\mathscr{G},
		\]
		\item 
		the pullback homomorphism 
		\[
		p^A\colon A^{*,*}(X)\to A^{*,*}(Y)
		\]
		is split injective.
	\end{enumerate}
\end{corollary}
\begin{proof}
	Let $Y=\mathscr{G}/N$ and $p\colon Y=\mathscr{G}/N \to X$ be the projection. Then it is clear that the $G$-torsor $p^*\mathscr{G}$ is induced from the $N$-torsor $\mathscr{G}\to \mathscr{G}/N=Y$ and we have item (1). For the item (2) note that Nisnevich local triviality of $\mathscr{G}$ yields that $p$ is a Nisnevich locally trivial fibration with the fiber $G/N$.	If $\chark k\neq 0$ or $G/N$ is strongly dualizable then Theorem~\ref{thm:main} yields that $\chi^{\A^1}(G/N)$ is a unit whence \cite[Theorem~1.10]{Lev19} gives the  claim. If $\chark k=p\neq 0$ then Theorem~\ref{thm:main} yields that $\chi^{\A^1}_{\Z[\frac{1}{p}]}(G/N)$ is a unit. Applying \cite[Theorem~1.10]{Lev19} to $A[\frac{1}{p}]$ we obtain that the pullback homomorphism $	A^{*,*}(X)\otimes_{\Z} \Z[\tfrac{1}{p}]\to A^{*,*}(Y)\otimes_{\Z} \Z[\tfrac{1}{p}]$	is injective. The claim follows since $A^{*,*}(X)$ has no $p$-torsion.
\end{proof}

\end{document}